\documentclass[12pt]{iopart}
\expandafter\let\csname equation*\endcsname\relax
\expandafter\let\csname endequation*\endcsname\relax

\usepackage{amsmath,amsfonts,amssymb,amsthm}
\usepackage{graphics,graphicx}
\usepackage{color,colordvi}
\usepackage{hyperref}
\usepackage{enumerate}
\usepackage{tikz}

\def\softd{{\leavevmode\setbox1=\hbox{d}%
\hbox to 1.05\wd1{d\kern-0.4ex{\char039}\hss}}}
\def\softt{{\leavevmode\setbox1=\hbox{t}%
\hbox to \wd1{t\kern-0.6ex{\char039}\hss}}}

\newcommand{\ee}{\mathrm{e}}
\newcommand{\R}{\mathbb{R}}

\newtheorem{theorem}{Theorem}[section]

\newtheorem{lemma}[theorem]{Lemma}
\newtheorem{corollary}[theorem]{Corrolary}
\newtheorem{remark}[theorem]{Remark}
\newtheorem{remarks}[theorem]{Remarks}

\begin{document}

\title[Magnetic transport in laterally coupled layers]
{Magnetic transport in laterally coupled layers}

\author{Pavel Exner$^{1,2}$}
\address{1) Nuclear Physics Institute, Czech Academy of Sciences,
Hlavn\'{i} 130, \\ 25068 \v{R}e\v{z} near Prague, Czech Republic}
\address{2) Doppler Institute, Czech Technical University, B\v{r}ehov\'{a} 7, 11519 Prague, \\ Czech Republic}
\ead{exner@ujf.cas.cz \\[-.3em] \phantom{...}}

\hspace{3.7em} \emph{\small To Igor Jex on the occasion of his 60th birthday}

\begin{abstract}
We discuss magnetic transport in the system of two adjacent hard-wall layers exposed to a homogeneous field perpendicular to the layer plane and coupled laterally through a strip-shaped window in the common boundary. We show that the spectrum is a combination of absolutely continuous and flat bands, the latter being present only if the widths of the two layers are commensurate, and derive their properties. We also analyze the one-sided geometry in which the barrier separating the two layers is a halfplane.
\end{abstract}


\vspace{2pc} \noindent{\it Keywords}: Dirichlet layers, lateral coupling, magnetic transport

%
\submitto{\PS}

\section{Introduction}
\label{s:intro}
\setcounter{equation}{0}
\setcounter{figure}{0}

Consider a particle whose motion is restricted to a plane or a layer. If the particle is charged and exposed to a homogeneous magnetic field perpendicular to such a configuration space region, its motion becomes localized, both classically and quantum mechanically. However, adding an obstacle we can make the particle to move along it and produce an electric current which even survives a weak disorder. This effect played an important role in explaining the quantum Hall effect following the seminal paper of Halperin \cite{Ha82}. The obstacle can be of various types, for instance, a potential barrier \cite{MMP99}, a wall with a boundary condition, Dirichlet \cite{DP99}, Neumann \cite{RS22}, or a mixed one; for more information we refer to the monograph~\cite{Ray}. Alternatively, an obstacle may be created by a local variation of the magnetic field as observed by Iwatsuka~\cite{Iw85}, see also \cite{MP97} or \cite[Sec.~6.5]{CFKS}.

And there are still other possibilities. It was shown recently \cite{EKT18} that the effect may also come from a geometric perturbation: edge currents arise if instead of a planar confinement the particle is kept within an appropriately bent layer. In the present paper we are going to describe another model in which the magnetic transport is due to the geometry of the system. A common feature with the previous example is that -- in contrast to the models mentioned above -- these geometrically induced edge currents \emph{have no classical counterpart} being thus of a purely quantum nature; it might not have been easy to see in \cite{EKT18} while here it will be quite obvious.

The model to discuss is simple. We consider two adjacent planar layers with hard-wall (Dirichlet) boundary in a homogeneous magnetic field perpendicular to them. The spectrum of such a system is pure point, the eigenvalues being sums of the Landau levels with the transverse mode energies. This changes when we couple the layers laterally by opening a window in the common boundary in the form of an infinite straight strip. Some of the previous eigenvalues may survive provided the widths of the two layers are commensurate but a infinite number of them spreads into absolutely continuous spectral bands representing the edge states describing transport along the window borders.

Let us describe briefly the contents of the paper. As a preparatory step we analyze in Sec.~\ref{s:neumann} a model example of a single layer with the Dirichlet condition replaced by Neumann in a strip at one of the boundaries and derive properties of its spectrum in dependence on the model parameters. This is first used in Sec.~\ref{s:latcoup} to describe a mirror-symmetric double layer, then we show there how the situation changes when the widths of two layers are different. Sec.~\ref{s:onesided} is devoted to the discussion of the one-sided geometry in which the boundary separating the two layers is removed in a whole halfplane and edge currents follow a single barrier border. Concluding remarks in Sec.~\ref{s:concl} present a few open questions about the model.

\section{A layer with an infinite Neumann window}
\label{s:neumann}
\setcounter{equation}{0}

To begin with, let us consider magnetic Laplacian in a planar layer $\Omega$ of width $d$ with the boundary which mixes Dirichlet and Neumann condition: it is Neumann at one of the boundaries in a straight strip of width $2a$ and Dirichlet in all the rest; the magnetic field is according to the assumption homogeneous and perpendicular to the layer, cf. Fig.~1. A simple symmetry consideration tells us that the indicated situation is equivalent to the (nontrivial part of the) spectral problem for a pair of adjacent layers of the same widths coupled laterally by an infinite strip-like `window'.

\begin{figure}
\begin{tikzpicture}
\hspace{3cm}
\draw[line width=1.5pt, fill=gray!45] (1,1) -- (3,3) -- (5,3) -- (3,1) -- cycle;
\draw[line width=1.5pt, fill=gray!15] (3,1) -- (5,3) -- (6,3) -- (4,1) -- cycle;
\draw[line width=1.5pt, fill=gray!45] (4,1) -- (6,3) -- (8,3) -- (6,1) -- cycle;
\draw[line width=1.5pt, fill=gray!45] (2,1) -- (1,0) -- (6,0) -- (8,2) -- (7,2) -- (6,1) -- cycle;
\draw[line width=.5pt, ->] (4.3,1.8) -- (4.3,4.2) node[left] {$z$};
\draw[line width=.5pt] (4.3,-.5) -- (4.3,0);
\draw[line width=.5pt] (4.3,.8) -- (4.3,1);
\draw[line width=.5pt, ->] (.8,.8) -- (7.7,.8) node[below right] {$x$};
\draw[line width=.5pt, ->] (6.5,3) -- (6.8,3.3) node[right] {$y$};
\draw[line width=.5pt] (3,-.5) -- (4.5,1);
\draw[line width=2pt, ->] (2,2.5) -- (2,4) node[below right] {${\vec B}$};
\draw (1.8,1.3) node{D};
\draw (3.8,1.3) node{N};
\draw (1.8,.3) node{D};
\end{tikzpicture}
\caption{Magnetic layer with a Neumann window}
\end{figure}
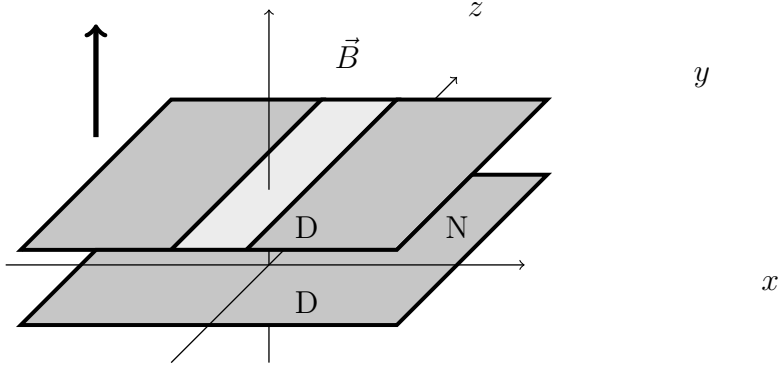

To put the description in mathematical terms, the layer has the Cartesian product form, $\Omega:=\Sigma\times\mathbb{R}$, with the cross-section $\Sigma:=\mathbb{R}\times(0,d)$; the Neumann condition is imposed at $W:= \{\vec x=(x,y,d):\, x\in(-a,a),\, y\in\mathbb{R}\}$. The magnetic field is of the form $\vec B=(0,0,B)$; without loss of generality we may assume $B>0$ and choose the Landau gauge for the corresponding vector potential, $\vec A=(0,Bx,0)$. The Hamiltonian is then an operator on $L^2(\Omega)$ acting as
\begin{subequations}
\label{Hamilt}
\begin{equation} \label{Hamilta}
H = -\partial_x^2 + (-i\partial_y+Bx)^2 -\partial_z^2
\end{equation}
with the derivatives understood in the distributional sense and the domain
\begin{equation} \label{Hamiltb}
D(H) = \{ \psi\in H^2(\Omega):\, H\psi\in L^2(\Omega),\,\psi(\vec x)=0 \;\;\text{if}\;\; x\in\partial\Omega\setminus W,\; \partial_z\psi(\vec x)=0 \;\;\text{if}\;\; x\in W\}
\end{equation}
\end{subequations}
which is obviously a subset $H^2(\Omega)$; if we wish to stress the dependence on the parameters we denote the operator as $H_a$ or $H_{a,d}$. As usual in situations with a translational symmetry -- see, e.g., \cite[Sec.~6.5]{CFKS} or \cite{Iw85, EKT18} and references therein -- one can simplify the spectral analysis with the help of a partial Fourier transformation by which operator \eqref{Hamilt} is unitarily equivalent to the direct integral
\begin{subequations}
\label{dirint}
\begin{equation} \label{dirinta}
H = \int^\oplus_\mathbb{R} H(p)\,\mathrm{d}p,
\end{equation}
where the fiber operators act on $L^2(\Sigma)$ as
\begin{equation} \label{dirintb}
H(p) = -\partial_x^2 + (p+Bx)^2 -\partial_z^2
\end{equation}
with the domain
\begin{align}
D(H(p)) = \{ & \psi:\: H(p)\psi\in L^2(\Sigma):\, \psi(x,0)=0,\;\, x\in\mathbb{R},\;\; \psi(x,d)=0\;\;\text{if}\;\; |x|\ge a \;\;\nonumber \\ & \&\;\; \partial_z\psi(x,d)=0\;\;\text{if}\;\; |x|<a, \;\; x\in\mathbb{R}\}, \label{dirintc}
\end{align}
\end{subequations}
contained in $H^2(\Sigma)$ and independent of the momentum variable $p$.

Assume first that the window is absent, $a=0$. In that case variables in the fiber operator $H_0(p)$ separate and the spectrum is easily found; it consists of eigenvalues
\begin{subequations}
\label{freespec}
\begin{equation} \label{freeev}
\lambda_{n,m} = B(2n+1) + \big(\textstyle{\frac{\pi m}{d}}\big)^2,\quad n\in\mathbb{N}_0,\; m\in\mathbb{N}
\end{equation}
combining the Landau levels with the Dirichlet eigenvalues in the transverse direction. They are independent of the momentum $p$ and associated with the eigenfunctions
\begin{equation} \label{freeef}
\phi_{n,m}(x,z)= \sqrt{\textstyle{\frac{2}{d}}}\, h_n\big(x+\textstyle{\frac{p}{B}}\big)\, \sin \textstyle{\frac{\pi mz}{d}},
\end{equation}
where $h_n$ are oscillator eigenfunctions
\begin{equation} \label{oscill}
h_n(u) = \frac{1}{\sqrt{2^n n!}}\, \big(\textstyle{\frac{B}{\pi}}\big)^{1/4} \ee^{-Bu^2/4}\, H_n\big(\sqrt{B}u\big).
\end{equation}
\end{subequations}
The eigenvalues \eqref{freeev} are simple provided $\frac{Bd^2}{\pi^2} \not\in\mathbb{Q}$, in the opposite case they may have multiplicity two. We can arrange them in the ascending order into a sequence $\{\lambda_k:\,k\in\mathbb{N}\}$, with the multiplicity taken into account in the rational case; the corresponding indices will be written as $k(n,m)$ when necessary. Consequently, the eigenvalues of $H_0 = \int^\oplus_\mathbb{R} H_0(p)\,\mathrm{d}p$ are infinitely degenerate and given by \eqref{freeev} again.

Let us pass to the case of nontrivial coupling, $a>0$, and look how the presence of the window will change the spectrum \eqref{freespec}. Its purely discrete character will not be altered, so following general properties of direct integrals of operators \cite[Sec.~XIII.16]{RS} we have to find the eigenvalues $\lambda_k(p)$ of the fiber operator \eqref{dirintb} and their dependence on the momentum variable $p$; if we need to indicate the dependence on the parameters, we write $\lambda_k(p;a)$ or $\lambda_k(p;a,d)$, $\lambda_k(p;B)$, etc. We begin with several auxiliary results.
\begin{lemma} \label{l:monoton}
For any $p,p'\in\R$ and $a,a'\ge 0$ we have $H_{a'}(p')\le H_a(p)$ in the form sense whenever $(p-a,p+a) \subset (p'-a',p'+a')$
\end{lemma}
\begin{proof}
The quadratic form associated with $H_a(p)$ is
\begin{subequations}
\label{fibform}
\begin{equation} \label{fibfora}
h_a[\psi;p] = \int_\Sigma \big( |\nabla\psi(x,z)|^2 +(p+Bx)^2 |\psi(x,z)|^2\big)\,\mathrm{d}x \mathrm{d}z
\end{equation}
with the domain
\begin{equation}
Q(H_a(p)) = \{ \psi\in H^1(\Sigma):\,h_a[\psi;p]<\infty\,, \; \psi(x,0)=0,\;\, x\in\partial\Sigma\setminus W_0(a)\}, \label{fibformb}
\end{equation}
\end{subequations}
where $W_0(a):= \{\vec x=(x,d):\, x\in(-a,a)\}$. By a simple change of variables in \eqref{fibfora}, $h_a[\psi;p]$ is unitarily equivalent to the form acting as $h_a[\psi;0]$ on the domain of the type \eqref{fibformb} in which $W_0(a)$ is replaced $W_p(a):= \{\vec x=(x,d):\, x\in(p-a,p+a)\}$. The inclusion relation $W_p(a)\subset W_{p'}(a')$ then implies the result in analogy with Proposition~4 in \cite[Sec.~XIII.15]{RS}
\end{proof}
\begin{corollary} \label{c:bounds}
Under the assumption of the Lemma~\ref{l:monoton} we have $\lambda_k(p;a) \ge \lambda_k(p';a')$, in particular
\begin{equation} \label{upperb}
\lambda_k(\infty) \le \lambda_k(p;a) \le \lambda_k(0) \quad\text{for}\quad p\in\mathbb{R},\;k\in\mathbb{N},
\end{equation}
where $\lambda_k(0) := \lambda_k = \lambda_{n,m}$ for $k=k(n,m)$, independently of $p$, and
\begin{equation} \label{nowall}
\lambda_k(\infty) := B(2n+1) + \big(\textstyle{\frac{\pi m}{2d}}\big)^2,\quad n\in\mathbb{N}_0,\; m\in\mathbb{N}.
\end{equation}
\end{corollary}
\begin{proof}
Since the operator \eqref{Hamilt} and its fibers in the decomposition \eqref{dirinta} are obviously below bounded, the claim follows from the form version of the minimax principle \cite[Thm.~XIII.2]{RS}. The lower bound \eqref{nowall} corresponds to the 'infinitely wide window', that is, a layer with one boundary Neumann and the other Dirichlet.
\end{proof}

Next we are going to show that the upper bound in \eqref{upperb} is saturated asymptotically.
\begin{lemma} \label{l:asympt}
$\lim_{|p|\to\infty}\lambda_k(p;a) = \lambda_k$ holds for any $k\in\mathbb{N}$.
\end{lemma}
\begin{proof}
The simplest way to find the limit is to use bracketing \cite[Sec.~XIII.15]{RS}. To be specific, consider the limit $p\to\infty$, the other case is similar. We clearly have $H_\mathrm{N}(p) \le H(p) \le H_\mathrm{D}(p)$ where the estimating operators are obtained by modifying the domain \eqref{dirintc} by adding the Neumann or Dirichlet condition, respectively, at the segment $\{-a\}\times(0,d)$. The spectrum of each of the estimating operators is then the union of the spectra in the two parts of the strip $\Sigma$. It is clear that for a \emph{fixed} $k=k(n,m)$ and $p$ large enough the eigenvalues between which $\lambda_k(p)$ is squeezed refer to the part of $\Sigma$ in the \emph{left} halfplane, that is, the one without the window. Indeed, the spectral threshold in the other part is easily estimated from below by $(p-Ba)^2 +\big(\frac{\pi}{2d}\big)^2$ so that it exceeds $\lambda_k$ eventually as $p$ increases.

In this part of the operator, however, the variables separate and the transverse contribution to the eigenvalue, $\big(\frac{\pi m}{d}\big)^2$, is independent of $p$. Hence it is sufficient to check that the eigenvalues of the harmonic oscillator restricted to the interval $\big(-\infty,\frac{p}{B}-a\big)$ with the D/N condition at its endpoint both converge to $B(2n+1)$ as $p\to\infty$; this can be done, e.g., using Feynman-Hellmann formula \cite{CHS02}.
\end{proof}

What is more, the upper bound in \eqref{upperb} is sharp away from the asymptotics.
\begin{lemma} \label{l:sharp}
$\lambda_k(p;a) < \lambda_k$ holds for any $a>0$, $\,p\in\mathbb{R}$, and $k\in\mathbb{N}$.
\end{lemma}
\begin{proof}
Given $k=k(n,m)$, in view of the minimax principle it is sufficient to find a unit-norm function $\psi\in D(H(p)$ orthogonal to $\mathcal{H}_k^{(-)}$, the subspace spanned by the eigenfunctions \eqref{freeef} with indices $n',m'$ satisfying $k(n',m')<k(n,m)$, plus -- in case of $\lambda_k(p)$ of multiplicity two -- the `other' eigenfunction, such that
\begin{equation} \label{var-est1}
(\psi,H(p)\psi) < \lambda_{n,m}
\end{equation}
To this aim we choose a nonzero function $\varphi$ with the support, say, in $(-\frac12 a,\frac12 a)\times(\frac14 d,\frac34 d)$ which is orthogonal to the eigenfunctions spanning the subspace $\mathcal{H}_k^{(-)}$ and such that $\mathrm{Re}\,(\varphi, \psi_{n,m})\ne 0$. Note that there is enough room for choosing such a function, since restrictions of the spanning functions to the indicated rectangle form a finite-dimensional subspace in  $L^2(-\frac12 a,\frac12 a)\otimes L^2(\frac14,\frac34 d)$ the dimension of which is infinite. Then we put $\psi_\varepsilon := \psi_{n,m} + \varepsilon\varphi$ obtaining for the left-hand side of \eqref{var-est1} the expression
\begin{equation} \label{var-est2}
(\psi_\varepsilon,H(p)\psi_\varepsilon) = \lambda_{n,m} \big(1+2\varepsilon\mathrm{Re}\,(\varphi,\psi_{n,m})\big) + \varepsilon^2\|H(p)^{1/2}\psi_{n,m}\|^2.
\end{equation}
This gives the sought result because for small enough parameter $\varepsilon$ the linear term on the right-hand side of \eqref{var-est2} dominates over the quadratic one and choosing $\varepsilon$ of a proper sign we can satisfy the inequality \eqref{var-est1}.
\end{proof}
\begin{remark} \label{r:capacity}
{\rm There is an alternative way to prove this inequality: one has to combine the sharp Dirichlet monotonicity result of \cite{GZ94} with the fact that the capacity of the window segment is positive \cite[Sec.~5.2]{Ra}. Moreover, this argument shows that the the function $\lambda_k(p;\cdot)$ is decreasing in $[0,\infty)$.}
\end{remark}

With these preliminaries, we can state and prove our first main result:
\begin{theorem} \label{thm:nwindow}
Spectrum of the operator \eqref{Hamilt} is for any positive $a,\,d$, and $B$ purely absolutely continuous and has the band-and-gap structure.
\end{theorem}
\begin{proof}
The band structure follows from the decomposition \eqref{dirinta}, and in addition, Lemma~\ref{l:sharp} shows that the Lebesgue measure of the spectrum is nonzero. To learn more about its character, we have to know more about the regularity of the functions $\lambda_k(\cdot)$. To this aim we express the quadratic form \eqref{fibform} in the vicinity of a point $p_0\in\R$ as the form sum
\begin{equation} \label{pertform}
h_a[\psi;p] = h_a[\psi;p_0] + q[\psi],
\end{equation}
where the perturbation can be estimated in the following way
\begin{align}
q[\psi;p] &= (p-p_0)^2 \|\psi\|^2 + 2(\psi,(p-p_0)(p_0+Bx)\psi) \nonumber \\[.3em]
&\le (p-p_0)^2 (1+\delta^{-1}) \|\psi\|^2 + \delta \|(p_0+Bx)\psi\|^2 \label{formest} \\[.3em]
&\le (p-p_0)^2 (1+\delta^{-1}) \|\psi\|^2 + \delta h_a[\psi;p_0]. \nonumber
\end{align}
Since $\delta$ can be any positive number, inequality \eqref{formest} means that $q[\cdot]$ is infinitesimally form bounded by $h_a[\cdot;p_0]$. Consequently, $h_a[\cdot;p]$ given by \eqref{pertform} is an analytic family of type (B) in the sense of Kato \cite[Sec.~7.4]{Ka}. This in turn implies that the eigenvalues $\lambda_k(p)$ are real analytic functions of the momentum variable, and as such they could be constant on an open subset of $\R$ only if they were constant, however, this possibility is excluded in view of Lemmata~\ref{l:asympt} and \ref{l:sharp}. The absolute continuity of the spectrum then follows from \cite[Thm.~XIII.86]{RS}.
\end{proof}

The spectral character is not the only claim one can make about our Hamiltonian, its explicit form allows us to say more about it:
\begin{theorem} \label{thm:windowprop}
In the described situation, spectrum of the operator \eqref{Hamilt} has the following properties:
 \begin{enumerate}[(i)]
 \setlength{\itemsep}{0pt}
\item The upper endpoint of each spectral band coincides with one of the eigenvalues \eqref{freeev} whose indices can be thus used to label the bands (and gaps). \label{propa}
\item For a fixed $p\in\R$ and fixed $a,d>0$, each eigenvalue $\lambda_k(p)$ depends continuously on the field intensity $B$ away from zero. \label{propb}
\item For a fixed $p\in\R$ and fixed $a,B>0$, each eigenvalue $\lambda_k(p)$ depends continuously on the layer width $d$ away from zero. \label{propc}
\item For a fixed $p\in\R$ and fixed $d,B>0$, each eigenvalue $\lambda_k(p,a)$ is continuous and decreasing as a function of the window width $a$ in $[0,\infty)$. \label{propd}
\item For fixed $d,B>0$ and indices $n,m$, there is a positive $a_\mathrm{o}$ such that the gap below the band indexed $(n,m)$ is open for $a<a_\mathrm{o}$. \label{prope}
\item The eigenfunctions corresponding to $\lambda_k(p,a)$ satisfy $\phi_k(x,z;p,a) = \phi_k(-x,z;-p,a)$, in particular, the probability current $p|\phi_k(x,z;p,a)|^2$ changes sign under the mirror transformation with respect to the $(y,z)$ plane. \label{propf}
 \end{enumerate}
\end{theorem}
\begin{proof}
Claim \eqref{propa} follows from Corollary~\ref{c:bounds}. As for \eqref{propb}, the continuity of $B\mapsto \lambda_k(p;B)$ in the vicinity of a fixed $B_0>0$ can be checked in a way similar to the proof of the previous theorem. To estimate form difference $h_a[\cdot;p,B]-h_a[\cdot;p,B_0]$ in terms of the last form, it is sufficient to bound $(p-Bx)^2-(p-B_0x)^2$ for all $x\in\R$ by $\delta(p-B_0x)^2$ with $\delta<1$ to be able to use analytical perturbation theory \cite{Ka}; this is clearly possible as long as $|B-B_0|<B_0$. Moreover, it shows that the function $B\mapsto \lambda_k(p;B)$ is in fact real analytic. In a similar way one can check the continuity in claim \eqref{propc} using scaling in the $z$ direction.

Likewise, the continuity (in fact, real analyticity) away from zero in claim~\eqref{propd} is proved using scaling in the $x$ direction, the strict monotonicity was mentioned in Remark~\ref{r:capacity}. To show the continuity at zero, we have to proceed differently. From Lemma~\ref{l:monoton} we know that the function $a\mapsto \lambda_k(p;a)$ is for any $k$ non-increasing which in combination with Corollary~\ref{c:bounds} means that $\lim_{a\to 0}\lambda_k(p;a) \le \lambda_k$ exists. To check that the last inequality is in fact equality, we recall that the behavior of eigenvalues in the situation when a Neumann window in a Dirichlet boundary shrinks to a point is a classical problem ~\cite{Sw63, MNP84, Ga92a}. The results in these and related papers are typically formulated for elliptic operators on bounded regions, however, the presence of the oscillatory potential in \eqref{dirintb} makes the spectrum discrete and allows one to modify the reasoning to the present situation.

To prove \eqref{prope}, let us fix $k=k(n,m)$. It follows from the continuity of $\lambda_k(\cdot,a)$ and Lemma~\ref{l:asympt} that to a given $a>0$ there is $p_a$ such that $\lambda_k(p_a,a) = \min_{p\in\R}\lambda_k(p,a)$. The indicated gap is open provided $\lambda_k(p_a,a)>\lambda_{k-1}$. If we take an $a_1<a$, the dispersion curve $\lambda_k(\cdot,a_1)$ has a minimum at $p_{a_1}$ and by Corolary~\ref{c:bounds} we have $\lambda_k(p_{a_1},a_1) > \lambda_k(p_a,a)$. In this way for any sequence $\{a_n\}$ such that $a_n\to 0$ as $n\to\infty$ we get an increasing sequence $\{\lambda_k(p_{a_n},a_n)\}$. The gap would close for any $a>0$ if $\lim_{n\to\infty} \lambda_k(p_{a_n},a_n) \le \lambda_{k-1}$, however, this would contradict the fact that $\lim_{a\to 0}\lambda_k(p;a)=\lambda_k$ for any $p\in\R$. The claim~\ref{propf} is obvious from \eqref{dirintb}.
\end{proof}

\begin{remarks}
{\rm (a) We do not ask about the continuity of $\lambda_k(p;\cdot)$ at the $B=0$ because the spectral character changes in the limit of vanishing magnetic field. The variables in the operator \eqref{Hamilta} then separate. The transverse part in the $(x,z)$-plane has then the essential spectrum covering the interval$\big(\big(\frac{\pi}{d}\big)^2,\infty\big)$ and a finite number of positive discrete eigenvalues below the threshold \cite[Sec.~1.5.1]{EK}. Consequently, the full operator has an absolutely continuous spectrum in which the states with the energy support below $\big(\frac{\pi}{d}\big)^2$ can only propagate being localized in the vicinity of the window. \\[.2em]
(b) A naive limit $d\to 0$ makes no sense, of course. One may ask, however, what happens if we renormalize the energy by subtracting the divergent quantity $\big(\frac{\pi}{d}\big)^2$. It might then happen that, in the spirit of \cite{Gr08}, the limit would be for a fixed $p\in\R$ generically trivial, but the question is not simple and we leave it open. \\[.2em]
(c) Using the machinery mention in the proof of \eqref{propd}, in particular, the method of matching asymptotic expansions \cite{Ga92b} one can not only prove the continuity but also to find the asymptotic behavior of $\lambda_k(p,a)$ in the narrow-window regime, $a\to 0$. We postpone discussion of this question to a subsequent publication.
}
\end{remarks}

\section{Laterally coupled layers}
\label{s:latcoup}
\setcounter{equation}{0}

Let us pass now to the primary subject of this paper, spectral properties of the magnetic Laplacian describing a charged particle confined to a pair of parallel and adjacent planar layers $\Omega_j,\, j=1,2,$ of widths $d_1, d_2$, in general different, coupled laterally through a straight window in the form of an infinite strip of width $2a$. The magnetic field is again supposed to be homogeneous, perpendicular to the layers, and pointing upwards as sketched in Fig.~2.

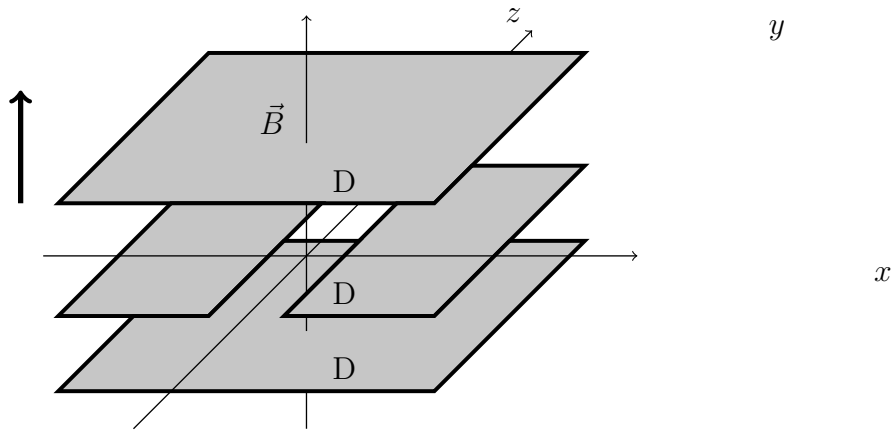
\begin{figure}
\begin{tikzpicture}
\hspace{3cm}
\draw[line width=1.5pt, fill=gray!45] (1,2.5) -- (3,4.5) -- (8,4.5) -- (6,2.5) -- cycle;
\draw[line width=1.5pt, fill=gray!45] (1,1) -- (2.5,2.5) -- (4.5,2.5) -- (3,1) -- cycle;
\draw[line width=1.5pt, fill=gray!45] (4,1) -- (5.5,2.5) -- (6,2.5) -- (6.5,3) -- (8,3) -- (6,1) -- cycle;
\draw[line width=1.5pt, fill=gray!45] (1,0) -- (2,1) -- (3,1) -- (4,2) -- (5,2) -- (4,1) -- (6,1) -- (7,2) -- (8,2) -- (6,0) -- cycle;
\draw[line width=.5pt, ->] (4.3,3.3) -- (4.3,5) node[left] {$z$};
\draw[line width=.5pt] (4.3,-.5) -- (4.3,0);
\draw[line width=.5pt] (4.3,.8) -- (4.3,1);
\draw[line width=.5pt] (4.3,1.3) -- (4.3,2.5);
\draw[line width=.5pt, ->] (.8,1.8) -- (8.7,1.8) node[below right] {$x$};
\draw[line width=.5pt, ->] (7,4.5) -- (7.3,4.8) node[right] {$y$};
\draw[line width=.5pt] (2,-.5) -- (5,2.5);
\draw[line width=2pt, ->] (.5,2.5) -- (.5,4) node[below right] {${\vec B}$};
\draw (1.8,1.3) node{D};
\draw (1.8,.3) node{D};
\draw (1.8,2.8) node{D};
\end{tikzpicture}
\caption{Laterally coupled magnetic layers}
\end{figure}

To fix the notation, the double layer is of the form, $\Omega:=\Sigma\times\mathbb{R}$, with the cross-section $\Sigma:=\mathbb{R}\times(-d_2,d_1) \setminus \{(x,0):\, |x|\ge a\}$; the Dirichlet condition is imposed at the `outer' boundary, $z=-d_2, d_1$, and at the plane $z=0$ except the window, $W:= \{\vec x=(x,y,0):\, x\in(-a,a),\, y\in\mathbb{R}\}$. The magnetic field is of the form $\vec B=(0,0,B)$ with $B>0$; we again choose the Landau gauge for the corresponding vector potential, $\vec A=(0,Bx,0)$. The Hamiltonian is then an operator on $L^2(\Omega)$ acting as
\begin{subequations}
\label{Hamilt2}
\begin{equation} \label{Hamilt2a}
H = -\partial_x^2 + (-i\partial_y+Bx)^2 -\partial_z^2
\end{equation}
with the domain
\begin{equation} \label{Hamilt2b}
D(H) = \{\psi\in H^2(\Omega) \cap H^1_0(\Omega):\: H\psi\in L^2(\Omega)\},;
\end{equation}
\end{subequations}
we again may write $H_a$ or $H_{a,d}$ when we wish to stress the dependence on the parameters. By a partial Fourier transformation the operator \eqref{Hamilt2} is unitarily equivalent to the direct integral
\begin{subequations}
\label{dirint2}
\begin{equation} \label{dirint2a}
H = \int^\oplus_\mathbb{R} H(p)\,\mathrm{d}p,
\end{equation}
where the fiber operators act on $L^2(\Sigma)$ as
\begin{equation} \label{dirint2b}
H(p) = -\partial_x^2 + (p+Bx)^2 -\partial_z^2
\end{equation}
with the domain
\begin{equation}
D(H(p)) = \{\psi\in H^2(\Sigma) \cap H^1_0(\Sigma):\: H(p)\psi\in L^2(\Sigma)\}, \label{dirint2c}
\end{equation}
\end{subequations}
independent of the momentum variable $p$. If the layers are decoupled, $a=0$, the variables in the fiber operator $H_0(p)$ separate and the spectrum  consists of the eigenvalues
\begin{subequations}
\label{freespec2}
\begin{equation} \label{freeev2}
\lambda_{n,m_1,m_2} = B(2n+1) + \big(\textstyle{\frac{\pi m_1}{d_1}}\big)^2 + \big(\textstyle{\frac{\pi m_2}{d_2}}\big)^2,\quad n\in\mathbb{N}_0,\; m_1,m_2\in\mathbb{N}
\end{equation}
combining the Landau levels with pairs of Dirichlet eigenvalues in the transverse direction. They are independent of the momentum $p$ and associated with the eigenfunctions
\begin{equation} \label{freeef2}
\phi_{n,m}(x,z) = \sqrt{\textstyle{\frac{2}{d}}}\, h_n\big(x+\textstyle{\frac{p}{B}}\big)\, \sin \textstyle{\frac{\pi mz}{d_j}},\quad (-1)^{j-1}z \in(0,d_j),\; j=1,2,
\end{equation}
\end{subequations}
where $h_n$ are oscillator eigenfunctions \eqref{oscill}. The eigenvalues \eqref{freeev2} may be simple or degenerate of maximum multiplicity three depending on rational (in)dependencies between the numbers $B$ and $\big(\frac{\pi}{d_j}\big)^2$; we again arrange them in the ascending order into a sequence $\{\lambda_k:\,k\in\mathbb{N}\}$, with $k=k(n,m_1,m_2)$ and the multiplicity taken into account in case of degeneracies. Needles to say, the eigenvalues of $H_0 = \int^\oplus_\mathbb{R} H_0(p)\,\mathrm{d}p$ are all infinitely degenerate and given by \eqref{freeev2}.

Passing to coupled layers, let us begin with the symmetric situation where the widths are the same, $d_1=d_2=:d$. The mirror symmetry makes than the analysis easy: the operator \eqref{dirint2a} and its fibers \eqref{dirint2b} commute with the reflection changing the sign of $z$ on $\Omega$ and $\Sigma$, respectively, hence the operators are reduced by projections on the $z$-even and $z$-odd subspaces. Spectral properties of the former coincide with those discussed in the previous section while the latter refer to an unperturbed single layer. We thus arrive at the following conclusions:
\begin{theorem} \label{thm:symmlayer}
In the described situation, spectrum of the operator \eqref{Hamilt2} consists of infinitely degenerate eigenvalues (= flat bands) \eqref{freeev} and absolutely continuous bands adjacent from below to them; both can be labeled as in Theorem~\ref{thm:nwindow} with an additional label indicating the even and odd parts. Furthermore, we have:
 \begin{enumerate}[(i)]
 \setlength{\itemsep}{0pt}
\item In the limit $a\to 0$ the spectrum shrinks to the family of flat bands \eqref{freeev}. \label{prop2a}
\item The eigenvalues $\lambda_k(p)$ determining the absolutely continuous bands depend continuously on the (positive values of the) parameters $p,\,B,\,a$, and $d$. \label{prop2b}
\item For fixed $d,B>0$ and fixed band index $k(n,m)$, the gap below the corresponding pair of bands is open for all $a$ small enough. \label{prop2c}
\item The eigenfunctions corresponding to $\lambda_k(p)$ satisfy $\phi_k(x,z;p) = \phi_k(-x,z;-p)$ so that the probability current $p|\phi_k(x,z;p)|^2$ changes sign under the mirror transformation with respect to the $x=0$ plane. \label{prop2d}
 \end{enumerate}
\end{theorem}

Claim \eqref{prop2d} reveals the characteristic feature of transport in coupled layers. In the non-magnetic analogue of our problem a transport exist too, since the variables describing the motion in $\Sigma$ and in the $y$ direction separate. The eigenfunctions referring to the discrete spectrum of the two-dimensional problem \cite[Sec.~1.5.1]{EK} combine with the free motion, $\ee^{\pm ipy}$, perpendicular to the $(x,z)$ plane. However, the said eigenfunctions (and generalized eigenfunctions) are symmetric or antisymmetric with respect to the line $x=0$, and as a consequence, the probability current associated with a nonzero $p$ is even with respect to the $(y,z)$ plane. In contrast, the magnetic transport exhibits a preferred direction: the mirror image of its profile corresponds to the motion in the opposite direction. If you wish this can be interpreted as a particular type of $\mathcal{PT}$-symmetry in which the space reflection and momentum reflection coming from time reversal are applied in two perpendicular directions.

We also stress that, in contrast to planar regions with Dirichlet boundary \cite[Sec.~7.2.1]{EK}, the magnetic transport in coupled layers we are discussing here has \emph{no meaningful classical analogue} as indicated in the introduction. The reason is that the set of initial conditions for which a charged particle can propagate by repeated reflections from the window edges has zero measure in the corresponding phase space.

Let us pass to the general case in which the widths $d_1$ and $d_2$ need not coincide.
\begin{theorem} \label{thm:asymmlayer}
In this case the spectrum of \eqref{Hamilt2} has a band-and-gap structure with bands labeled by the indices appearing in \eqref{freeef2}, plus possibly an additional label specified in \eqref{prop22a} below, and the following properties:
 \begin{enumerate}[(i)]
 \setlength{\itemsep}{0pt}
\item The spectrum is absolutely continuous if $d_1$ and $d_2$ are incommensurate. In the opposite case there is an infinite number of flat bands corresponding to pairs $(m_1,m_2)$ satisfying $\frac{m_1}{m_2}=\frac{d_1}{d_2}$; to each of them there is an absolutely continuous band adjacent to it from below. \label{prop22a}
\item In the limit $a\to 0$ the spectrum shrinks to the family of flat bands \eqref{freeev2}. \label{prop22b}
\item The eigenvalues $\lambda_k(p)$ determining the absolutely continuous bands depend continuously on the (positive values of the) parameters $p,\,B,\,a$, and $d$. \label{prop22c}
\item For fixed $d,B>0$ and fixed band index $k(n,m_1,m_2)$, the gap below the corresponding band, or pair of bands, is open for all $a$ small enough. \label{prop22d}
\item The eigenfunctions corresponding to $\lambda_k(p)$ defining an absolutely continuous band satisfy $\phi_k(x,z;p) = \phi_k(-x,z;-p)$ so that the probability current $p|\phi_k(x,z;p)|^2$ changes sign under the mirror transformation with respect to the $x=0$ plane. \label{prop22e}
 \end{enumerate}
\end{theorem}
\begin{proof}
Each operator $H(p)$ has a discrete spectrum which follows from the minimax principle and the fact that it is bounded from below by the $H(p)$ with the barrier removed, $a=\infty$, which has eigenvalues $B(2n+1) + \big(\textstyle{\frac{\pi m}{d_1+d_2}}\big)^2$ of multiplicity at most two. In view of \eqref{dirint2a}, the spectrum of $H$ has a band structure.

If $\frac{m_1}{m_2}=\frac{d_1}{d_2}$, the eigenfunctions \eqref{freeef2} vanish at the the line $z=0$, and since they have there the same partial derivative with respect to $z$, being thus locally odd and smooth, they belong to $D(H(p))$ and give rise to eigenfunctions, even when the window is open, producing the flat bands. In the decoupled case, $a=0$, however, such an eigenvalue has multiplicity of at least two, and other linear combinations of the two parts are not smooth anymore, thus they are affected by the coupling and give rise to nonzero width bands. Needles to say, if $d_1$ and $d_2$ are incommensurate, the eigenfunctions \eqref{freeef2} do not vanish at the the dividing line and are affected by the coupling.

Mimicking then step by step the reasoning that lead to Theorem~\ref{thm:nwindow} one can check that for a given $k=k(n,m_1,m_2)$ the function $\lambda_k(\cdot)$ is real analytic with $\lambda_k(p)<\lim_{|p|\to\infty}\lambda_k(p)=\lambda_k$, and that for the other parameters fixed it is decreasing as function of $a$. This yields the absolute continuity of the non-flat bands. The continuity of $\lambda_k(p)$ with respect to the parameters $B\,\,a,\,d_1$, and $d_2$ which concludes the proof of claim~\eqref{prop22c} can be checked in a similar way as in Sec.~\ref{s:neumann}.

To prove claim~\eqref{prop22d}, we use bracketing estimating our operator $H(p)$ from below by the Hamiltonian of the system in which the window is replaced by Neumann condition. By minimax principle the eigenvalues of $H(p)$ are thus squeezed between those of such a Neumann modified Hamiltonian and the one without the coupling window, however, by Theorem~\ref{thm:nwindow}\eqref{propd} the two bounds converge to each other as $a\to 0$. The remaining claim \eqref{prop22e} is obvious from \eqref{Hamilt2a}.
\end{proof}

\section{One-sided barrier}
\label{s:onesided}
\setcounter{equation}{0}

If the window width $2a$ is large the transport at one of its edges is practically independent of what is happening at the other one. To understand it better, it is useful to fix the position of the edge and to examine what happens if the other disappears, that is, to assume that one side of the barrier between the layers is removed as sketched in Fig.~3.
\begin{figure}
\begin{tikzpicture}
\hspace{3cm}
\draw[line width=1.5pt, fill=gray!45] (1,2.5) -- (3,4.5) -- (8,4.5) -- (6,2.5) -- cycle;
\draw[line width=1.5pt, fill=gray!45] (1,1) -- (2.5,2.5) -- (5,2.5) -- (3.5,1) -- cycle;
\draw[line width=1.5pt, fill=gray!45] (1,0) -- (2,1) -- (3.5,1) -- (4.5,2) -- (5,2) -- (7,2) -- (8,2) -- (6,0) -- cycle;
\draw[line width=.5pt, ->] (4.3,3.3) -- (4.3,5) node[left] {$z$};
\draw[line width=.5pt] (4.3,-.5) -- (4.3,0);
\draw[line width=.5pt] (4.3,.8) -- (4.3,2.5);
\draw[line width=.5pt, ->] (.8,1.8) -- (8.7,1.8) node[below right] {$x$};
\draw[line width=.5pt, ->] (7,4.5) -- (7.3,4.8) node[right] {$y$};
\draw[line width=.5pt] (2,-.5) -- (5,2.5);
\draw[line width=2pt, ->] (.5,2.5) -- (.5,4) node[below right] {${\vec B}$};
\draw (1.8,1.3) node{D};
\draw (1.8,.3) node{D};
\draw (1.8,2.8) node{D};
\end{tikzpicture}
\caption{Straight interface between split and non-split magnetic layer}
\end{figure}
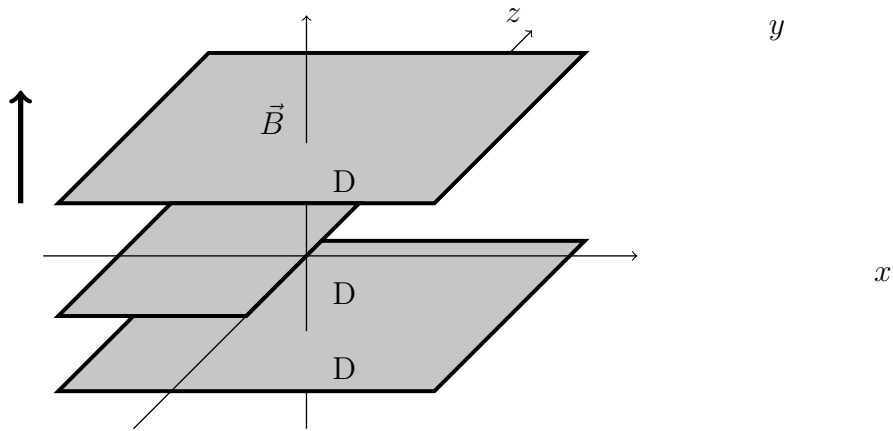

The double layer is now of the form, $\Omega:=\Sigma\times\mathbb{R}$, with the cross-section $\Sigma:=\mathbb{R}\times(-d_2,d_1) \setminus \{(x,0):\, x\ge 0\}$; the Dirichlet condition is imposed at the `outer' boundary, $z=-d_2, d_1$, and at the left part of the halfplane $z=0$. The magnetic field is of the same form as before corresponding to the vector potential, $\vec A=(0,Bx,0)$. The Hamiltonian acts on $L^2(\Omega)$ as
\begin{subequations}
\label{Hamilt3}
\begin{equation} \label{Hamilt3a}
H = -\partial_x^2 + (-i\partial_y+Bx)^2 -\partial_z^2
\end{equation}
with the domain
\begin{equation} \label{Hamilt3b}
D(H) = \{\psi\in H^2(\Omega) \cap H^1_0(\Omega):\: H\psi\in L^2(\Omega)\},
\end{equation}
\end{subequations}
and by partial Fourier transformation, \eqref{Hamilt3} is unitarily equivalent to the direct integral
\begin{subequations}
\label{dirint3}
\begin{equation} \label{dirint3a}
H = \int^\oplus_\mathbb{R} H(p)\,\mathrm{d}p,
\end{equation}
where the fiber operators act on $L^2(\Sigma)$ as
\begin{equation} \label{dirint3b}
H(p) = -\partial_x^2 + (p+Bx)^2 -\partial_z^2
\end{equation}
with the domain independent of $p$, namely
\begin{equation}
D(H(p)) = \{\psi\in H^2(\Sigma) \cap H^1_0(\Sigma):\: H(p)\psi\in L^2(\Sigma)\}. \label{dirint3c}
\end{equation}
\end{subequations}
Its spectrum is again purely discrete, the difference to the previous sections is that now we do not have a parameter controlling the perturbation allowing us, in particular, to switch it off. Nevertheless, a significant part of the previous reasoning can be adapted to the present situation.
\begin{theorem} \label{thm:onesided}
The spectrum of \eqref{Hamilt3} consists of bands, in general overlapping, with the following properties:
 \begin{enumerate}[(i)]
 \setlength{\itemsep}{0pt}
\item The spectrum is absolutely continuous if $d_1$ and $d_2$ are incommensurate. In the opposite case there is an infinite number of flat bands at the values \eqref{freeev2} with $(m_1,m_2)$ satisfying $\frac{m_1}{m_2}=\frac{d_1}{d_2}$; to each of them there is an absolutely continuous band adjacent to it from below. \label{prop3a}
\item The lower edges of the absolutely continuous are of the form
$$ 
\lambda_{n,m}^\mathrm{free} = B(2n+1) + \big(\textstyle{\frac{\pi m}{d_1+d_2}}\big)^2,\quad n\in\mathbb{N}_0,\; m\in\mathbb{N},
$$ 
in particular, $\inf\sigma(H)= B+\big(\textstyle{\frac{\pi}{d_1+d_2}}\big)^2$ and the spectrum in the vicinity of the threshold is absolutely continuous. \label{prop3b}
\item If the layer widths $d_1,\,d_2$ are unequal and $B\ge 3\big(\textstyle{\frac{\pi}{d_1+d_2}}\big)^2$, the spectrum contains an open gap. \label{prop3c}
 \end{enumerate}
\end{theorem}
\begin{proof}
The band character of the spectrum follows from the fact that the spectrum of each $H(p)$ is purely discrete similarly as in the proof of Theorem~\ref{thm:asymmlayer}. If $d_1$ and $d_2$ are rationally related, the argument used to prove claim \eqref{prop22a} there shows the existence of the flat bands in \eqref{prop3a} and of another eigenvalue $\lambda_k(p)$ of $H(p)$ below them for each $p\in\R$ and $k=k(n,m_1,m_2)$. Repeating the reasoning from the proof of Theorem~\ref{thm:nwindow} we infer that the functions $\lambda_k(\cdot)$ are real analytic, hence the non-flat bands are absolutely continuous.

Furthermore, by a simple change of variable operator \eqref{dirint3b} is unitarily equivalent to $-\partial_u^2 + B^2u^2 -\partial_z^2$ with barrier at the halfline $\big(-\infty,\frac{p}{B}\big)$; using the argument from Remark~\ref{r:capacity} we conclude that each $\lambda_k(\cdot)$ is an increasing function. As in Lemma~\ref{l:asympt} we have $\lim_{p\to\infty} \lambda_k(p)=\lambda_k$ given by \eqref{freeev2} and similarly we get
\begin{equation} \label{lowerlim}
\lim_{p\to-\infty} \lambda_k(p) = \lambda_{n',m'}^\mathrm{free}
\end{equation}
for some $(n',m')$ such that that $\lambda_{n',m'}^\mathrm{free}<\lambda_k$; the map $k\mapsto (n',m')$ is injective with the possible exception of the situation when $B$ and $\big(\textstyle{\frac{\pi}{d_1+d_2}}\big)^2$ are rationally related and the spectrum of the double layer without the barrier is not simple.

This completes the proof of claim~\eqref{prop3b} because the lower edge of the first band corresponds to $n=0$ and $m=1$ and the nearest possible flat band, if it exists, must be above $B+\big(\textstyle{\frac{\pi}{\min(d_1,d_2)}}\big)^2$ since the index combination $m_1=m_2=1$ is excluded if $d_1\ne d_2$. While most bands touch or overlap, open gaps still may exist. Consider the values \eqref{lowerlim} with $n=0$. Under the condition in~\eqref{prop3c} the second lowest one corresponding to $m=2$ does not exceed any value referring to $n\ge 1$ and at the same time, the upper edge of the first band is at
$$ 
\lambda_2 = \lambda_{0,1,1} = B+\Big(\frac{\pi}{\min(d_1,d_2)}\Big)^2 < B+\Big(\frac{2\pi}{d_1+d_2}\Big)^2,
$$ 
where the last inequality holds whenever $d_1\ne d_2$; this proves claim~\eqref{prop3c}.
\end{proof}

The preferred direction character of the magnetic transport is again visible. The eigenfunctions of $H(p)$ with $p>0$ giving rise to the generalized eigenfunctions of $H$ describing transport in the direction of the $y$ axis are dominantly supported in the divided part of the double layer, the more the larger $p$ is, while the transport in the opposite direction is more pronounced in the undivided part. This is a rough description, however, which does not take into account effects in the transport associated with higher Landau levels coming from the shape of the oscillator eigenfunctions~\ref{oscill}.

\section{Concluding remarks}
\label{s:concl}
\setcounter{equation}{0}

The above discussion by far does not exhaust all questions one may ask about properties of these systems. We have touched already some; let us mention other directions in which analysis of the model could be extended:
\begin{enumerate}[(a)]
\setlength{\itemsep}{0pt}
\item \emph{Weak coupling:} We know that the bands shrink to points as $a\to 0$, asymptotic expansion of the functions $\lambda_k(\cdot)$ in the vicinity of this point would allow us to understand better the dependence of the band widths on the parameters and make conclusions about the number of open gaps in the spectrum.
\item \emph{Properties of the \emph{ac} spectrum:} The tools used here do not allow to see the finer structure of the dispersion functions, in particular, the character of their crossings, their derivatives, etc. \label{qa}
\item \emph{Other lateral couplings:} One can ask about effects caused by windows of different shapes. As in the present case, the analysis simplifies if we can reformulate it as a problem of reduced dimension, say, for parallel strip windows, an annular form, etc. There are other interesting cases, though, like non-straight strip windows or crossed strips which require a different technique.
\item \emph{Other geometric perturbations:} In \cite{EKT18} we have shown that a magnetic transport can occur in a single Dirichlet layer if it is bent appropriately keeping the translational invariance in one dimension. In the same vein one can ask, e.g., what happens in layers, single or laterally coupled, if their width is locally modified.
\item \emph{Impurity effects:} A big question in transport problems generally, and magnetic ones in particular, is the stability with respect to random perturbations, see for instance \cite{DP99, FGW00, CHS02, HS08a, HS08b}. Here too one expects that the disorder will give rise to Anderson localization, but a part of the absolutely continuous may survive if the added potential responsible for it will be sufficiently weak. We note that this question is closely connected with \eqref{qa} above.
\end{enumerate}

\subsection*{Acknowledgment}

The research was supported by the Czech Science Foundation within the project 21-07129S and by the EU project CZ.02.1.01/0.0/0.0/16\textunderscore 019/0000778.

\subsection*{ORCID iDs}

Pavel Exner  https://orcid.org/0000-0003-3704-7841



\subsection*{References}


\begin{thebibliography}{ABCD}
\bibitem[dBP99]{DP99}
S.~de Bi\'evre, J.V.~Pul\'e: Propagating edge states for a magnetic Hamiltonian, \emph{Math. Phys. El. J.} \textbf{5} (1999), 3.
\bibitem[CHS02]{CHS02}
J.-M.Combes, P.~Hislop, \'{E}.~Soccorsi Edge states for quantum Hall Hamiltonians, in \emph{Mathematical Results in Quantum Mechanics}, Conteporary Math., vol.~307, AMS 2002; pp.~69--91.
\bibitem[CFKS]{CFKS}
H.L.~Cycon, R.G.~Froese, W.~Kirsch, B.~Simon: \emph{ Schr\"odinger Operators, with Applications to Quantum Mechanics and Global
Geometry}, Springer, Berlin and Heidelberg 1987.
\bibitem[EKT18]{EKT18}
P.~Exner, T.~Kalvoda, M.~Tu\v{s}ek: A geometric Iwatsuka type effect in quantum layers, \emph{J. Math. Phys.} \textbf{59} (2018), 042105 (19pp)
\bibitem[EK]{EK}
P.~Exner, H.~Kova\v{r}\'{\i}k: \emph{Quantum Waveguides}, Springer International, Heidelberg 2015.
\bibitem[FGW00]{FGW00}
J. Fr\"ohlich, G.M. Graf, J. Walcher: On the extended nature of edge states of quantum Hall Hamiltonians, \emph{Ann. Henri Poincar\'e} \textbf{1} (2000), 405--442.
\bibitem[Ga92a]{Ga92a}
R.R.~Gadyl'shin: Ramification of a multiple eigenvalue of the Dirichlet problem for the Laplacian under singular perturbation of the boundary condition, \emph{Math. Notes} \textbf{52} (1992), 1020---1029.
\bibitem[Ga92b]{Ga92b}
R.R.~Gadyl'shin: Surface potentials and the method of matching asymptotic expansions in the Helmholtz resonator problem, Algebra i Analiz, 1992, Volume 4, Issue 2, 88--115.
\bibitem[GZ94]{GZ94}
F.~Gesztesy, Z.~Zhao: Domain perturbations, Brownian motion, capacities, and ground states of Dirichlet Schr\"odinger
operators, \emph{Math. Z.} \textbf{215} (1994), 143--150.
\bibitem[Gr08]{Gr08}
D.~Grieser:  Spectra of graph neighborhoods and scattering, \emph{Proc.~London Math.~Soc.} \textbf{97} (2008), 718--752.
\bibitem[Ha82]{Ha82}
B.I. Halperin: Quantized Hall conductance, current carrying edge states, and the existence of extended states in two-dimensional disordered potential, \emph{Phys. Rev. B} \textbf{25} (1982), 2185--2190.
\bibitem[HS08a]{HS08a}
P.~Hislop, \'{E}.~Soccorsi: Edge currents for quantum Hall systems. I. One-edge, unbounded geometries, {\em Rev. Math. Phys.} {\bf 20} (2008), 71--115.
\bibitem[HS08b]{HS08b}
P.~Hislop, \'{E}.~Soccorsi: Edge currents for quantum Hall systems. II. Two-edge, bounded and unbounded geometries, {\em Ann. Henri Poincar\'e} {\bf 9} (2008), 1141--1175.
\bibitem[Iw85]{Iw85}
A. Iwatsuka: Examples of absolutely continuous Schr\"odinger operators in magnetic fields, \emph{Publ. RIMS} \textbf{21} (1985), 385--401.
\bibitem[Ka]{Ka}
T.~Kato: \emph{Perturbation Theory for Linear Operators}, 2nd edition, Springer, Berlin 1976.
\bibitem[MMP99]{MMP99}
N. Macris, Ph.A. Martin, J. Pul\'e: On edge states in semi-infinite quantum Hall systems, \emph{J. Phys. A: Math. Gen.} \textbf{32} (1999), 1985-1996.
\bibitem[MP97]{MP97}
M.~Mantoiu, R.~Purice: Some propagation properties of the Iwatsuka model, \emph{Commun. Math. Phys.} \textbf{188} (1997), 691--708.
\bibitem[MNP84]{MNP84}
V.G.~Maz'ya, S.A.~Nazarov, B.A.~Plamenevskii:  Asymptotic expansions of eigenvalues of boundary value problems for the Laplace operator in domain with small holes, \emph{Izv. Acad. Sci. USSR: Mat.} \textbf{48} (1984), 321--345.
\bibitem[Ran]{Ra}
T.~Ransford: \emph{Potential Theory in the Complex Plane}, Cambridge University Press 1995.
\bibitem[Ray]{Ray}
N.~Raymond: \emph{Bound States of the Magnetic Schr\"odinger Operator}, EMS Tracts in Mathematics, vol.~27, EMS, Z\"urich, 2017.
\bibitem[RS22]{RS22}
N.~Raymond, \'{E}.~Soccorsi: Magnetic quantum currents in the presence of a Neumann wall, \texttt{arXiv:2202.03710}
\bibitem[RS]{RS}
M.~Reed, B.~Simon: \emph{Methods of Modern Mathematical Physics, IV.~Analysis of Operators}, Academic Press, New York 1978.
\bibitem[Sw63]{Sw63}
C.A.~Swanson: Aymptotic variational formul{\ae} for eigenvalues, \emph{Canad. Math. Bull.} \textbf{6} (1963), 15--25.


\end{thebibliography}
\end{document}